\documentclass[12pt]{amsart}
\usepackage{amssymb}
\usepackage[latin1]{inputenc}
\usepackage{amsthm}
\usepackage{amsfonts}
\usepackage{mathrsfs}
\usepackage{dsfont}
\usepackage{mathtools}

\usepackage{srcltx}
\usepackage{amsmath}
\usepackage[all]{xy}

\let\mathcal\mathscr
\def\bA{{\mathbb A}}

\def\C{{\mathbb C}}

\def\bR{{\mathbb R}}

\def\bN{{\mathbb N}}

\newtheorem{thm}{Theorem}[section]

\def\P{{\bf P}}

\def\R{{\bf R}}
\def\Q{{\bf Q}}
\def\C{{\bf C}}

\def\Spec{\mathop{\rm Spec}\nolimits}

\def\tilde{\widetilde}

\def\phi{\varphi}

\def\cC{{\mathcal C}}

\def\cF{{\mathcal F}}

\def\cL{{\mathcal L}}

\def\cO{{\mathcal O}}

\def\cT{{\mathcal T}}
\def\cU{{\mathcal U}}

\def\cX{{\mathcal X}}

\def\cZ{{\mathcal Z}}

\def\div{\mathop{\rm div}\nolimits}

\def\Card{\mathop{\rm Card}\nolimits}

\numberwithin{equation}{section}

\newtheorem{theorem}[thm]{Theorem}

\newtheorem{claim}[thm]{Claim}

\newtheorem{corollary}[thm]{Corollary}

\newtheorem{definition}[thm]{Definition}

\newtheorem{lemma}[thm]{Lemma}

\newtheorem{proposition}[thm]{Proposition}

\begin{document}

\title {Rational vs transcendental points on analytic Riemann surfaces}
\author{Carlo Gasbarri}\address{Carlo Gasbarri, IRMA, UMR 7501
 7 rue Ren\'e-Descartes
 67084 Strasbourg Cedex}\date{}
\thanks{Research supported by the project ANR-16-CE40-0008 FOLIAGE}


\begin{abstract} Let $(X,L)$ be a polarized variety over a number field. We suppose that $L$ is an hermitian line bundle. Let $M$ be a non compact Riemann Surface and $U\subset M$ be a relatively compact open set. Let $\varphi:M\to X(\C)$ be a holomorphic map. For every positive real number $T$, let $A_U(T)$ be the cardinality of the set of  $z\in U$ such that $\varphi (z)\in X(K)$ and $h_L(\varphi(z))\leq T$. After a revisitation of the proof of the  sub exponential bound for $A_U(T)$, obtained by Bombieri and Pila ,   we show that there are intervals of $T$'s as big as we want for which $A_U(T)$ is upper bounded by a polynomial in $T$. We then introduce subsets of type $S$ with respect of $\varphi$. These are compact subsets of $M$ for which an inequality similar to Liouville inequality on algebraic points holds. We show that, if $M$ contains a subset of type $S$, then, {\it for every value of $T$} the number $A_U(T)$ is bounded by a polynomial in $T$. As a consequence, we show that if $M$ is a  smooth leaf of a foliation in curves then $A_U(T)$ is bounded by a polynomial in $T$. Let $S(X)$ be the subset (full for the Lebesgue measure) of points which verify some kind of Liouville inequalities.  In the second part we prove that $\varphi^{-1}(S(X))\neq\emptyset$ if and only if  $\varphi^{-1}(S(X))$ is full for the Lebesgue measure on $M$.\end{abstract}

\maketitle

\tableofcontents
\section {Introduction}

In the recent years there have been an interest in the following problem: Let $X$ be a projective variety defined over $\Q$ and $M\subset X(\C)$ be an analytic subset. Can we estimate, as a function of $T$,  the number of {\it rational} points with height less or equal then $T$ which are contained in $M$? A cornerstone in this topic is the paper by Bombieri and Pila \cite{BOMBIERIPILA} which essentially give a satisfactory answer in  the case when $M$ is a real curve in $\R^2$. After Bombieri and Pila's Theorem, a big amount of work on related problems has been done, culminating with the theorem by Pila and Wilkie \cite{pilawilkie} who could understand the problem in the case when $M$ is a definable subset in a suitable $o$--minimal structure. In many researches on this topic, the Bombieri and Pila results have been taken as a "black box" and applied to study more general cases. 

In this paper we would like to analyze again the basic case of rational points in  dimension one  sub varieties  contained a projective variety; we concentrate on the problem of estimating the number of points of bounded height in a Riemann surface which is contained, and Zariski dense, in a projective variety. 

Suppose that $X$ is a projective variety of dimension $n$ defined over $\Q$ (in order to simplify notations, in this introduction we will work over $\Q$, the general case will be treated in the paper). Fix an ample line bundle $L$ over $X$ and a height function $h_L(\cdot)$ associated to it. Let $M$ be a non compact Riemann surface and $U$ a relatively compact open set of it. Let $\varphi: M\to X(\C)$ be a holomorphic map with Zariski dense image. In this paper we are mostly interested in estimating from above, as a function of $T$,  the cardinality $A_U(T)$ of the set $S_U(T):=\{ z\in U \;/\; \varphi(z)\in X(\Q) \; {\rm and}\; h_L(\varphi(z))\leq T\}$.

The starting point of the paper is a revisitation of the Bombieri and Pila Theorem in the complex analytic setting. We give a self contained proof of their theorem which is more inspired to technics of analytic and diophantine geometry. We prove the following:

\begin{theorem}\label{bombieri pila intro} (Bombieri - Pila, cf. Theorem \ref{bombieri pila}) For every positive number $\epsilon$, we have
\begin{equation}
A_U(T)\ll \exp(\epsilon T)
\end{equation}
where the involved constants depend on $U$, $\varphi$, $\epsilon$ and $L$ but not on $T$. \end{theorem}

The proof follows only partially the strategy of Bombieri and Pila,  actually it is more related to classical proofs in transcendence: one covers  $U$ with $O(\exp(\epsilon T))$ small open sets of small area. Then we show that on each of these open set the cardinality of the points of height less or equal then $T$ is bounded by a constant. This is done by  constructing, via a form of Siegel Lemma and Liouville inequality, a section of {\it fixed degree} vanishing on each of these points.

There are examples in the literature which show that Theorem \ref{bombieri pila intro} is optimal. Never the less one may wonder for how many $T$'s, the number $A_U(T)$ is indeed big when compared to $T$. We discovered the following interesting fact: there are intervals $I$ as big as we want such that, for every $T\in I$, the number $A_U(T)$ {\it is smaller than a polynomial in $T$}. More precisely we proved:

\begin{theorem}\label{rare distribution 1intro} (cf. Theorem \ref{rare distribution 1}) Let  $A>1$, $\gamma>{{n}\over{n-1}}$ and $\epsilon>0$.  
Then the there are infinitely many intervals $I_t$ of the form $[t,At]$ such that,  for every $T\in I_t$ one has $A_U(T)\leq \epsilon T^\gamma$.  \end{theorem}
 
We also prove that the union of these $I_t$ is unbounded. The proof uses again a form of Siegel Lemma and of Liouville inequality and an argument by contradiction (which actually prevents to control the sets $I_t$). The reason why we cannot deduce a stronger version of Theorem \ref{bombieri pila intro} from the strategy of proof of Theorem \ref{rare distribution 1intro} relies on the use of the Liouville inequality: in order to obtain an upper bound  for the number of zeroes of an analytic function on $U$, we need an upper bound of the norm of it {\it and an lower bound of its norm on a specific point where it do not vanish}. This principle is explained in section  \ref{degree of a divisor}.

Actually, in order to obtain a good upper bound for $A_U(T)$, it suffices to know a lower bound of the norm of a section of a line bundle, not just on a point (where a priori it could vanish) but over a suitable subset. For this reason we introduce the following definition: we suppose that the holomorphic line bundle $L$ over $X(\C)$ is equipped with a smooth hermitian metric.

\begin{definition}\label{sets of type S intro} (cf. Definition \ref{sets of type S}) Let $B\subset U$ be a compact set and $a$ be a real number. We will say that $B$ is a subset of type $S_a$ of $M$  with respect to $\varphi$ if, we can find a positive constant $A>1$ such that, for every positive integer $d$ and $s\in H^0(\cX.\cL^d)\setminus\{0\}$ we have that
\begin{equation}
\log\Vert s\Vert_B\geq -A(\log^+\|s\|+d)^a.
\end{equation}
\end{definition}
And we prove
\begin{theorem}\label{poly T intro} (cf. Theorem \ref{poly T}) Suppose that we can find a subset $B\subset U$ of type $S_a$ with respect to $\varphi$. Then 
\begin{equation}
A_U(T)\ll T^{2a}.
\end{equation}
\end{theorem}

Of course, in general it is  not be easy to guarantee the existence of a subset of type $S_a$ with respect to $\varphi$. Never the less there is an important case of maps where we can guarantee the existence of such sets: the leaves of foliations. Suppose that $X$ is equipped with a foliation by curves $\cF$ (a priori not smooth). We can prove that every relatively compact open neighborhood of a rational point of a leaf is of type $S_a$ for a suitable $a$ (cf. after) and consequently we find:

\begin{theorem}\label{foliations intro} (cf. Theorem \ref{foliations}) Let $\cF$ be a foliation on a smooth quasi projective variety $Z$ defined over a number field $K$. Let $p\in Z(K)$ be a rational point and $h:\Delta_1\to Z_\sigma(\C)$ be the analytic leaf  of $\cF$ through $p$. Suppose that the dimension of the Zariski closure of the leaf is $\ell>1$. Let $0<r<1$ be a real number, then, for every $\epsilon>0$ we have
\begin{equation}
A_{\Delta_r}(T)\ll_\epsilon T^{2\ell+\epsilon}.
\end{equation}\end{theorem}
Here and after $\Delta_r$ will be the disk $\{ |z|<r\}$ in $\C$. In order to prove this, we require a Zero Lemma which have been proved by Nesterenko and generalized by Binyamini which holds for smooth points of foliations and for some kind of singularities of them. For this reason, the method could be generalized to the case some kind of singular points of the foliation but we do not think that it would be the case for a general singularity. 

The bigger the subset $B$ is and the easier should be to find it. But, in order to deduce the consequence on the number of rational points, it suffices that $B$ is non empty. In particular it can be a single point. In a previous paper \cite{gasbarri2}, we proved that the set $S_a(X)$ of  points of type $S_a$ in $X(\C)$ (definition similar to Definition \ref{sets of type S intro}, cf. Definition \ref{points of type S 1}) are full in $X(\C)$ for the Lebesgue measure as soon as $a\geq n+2$. In the last part of this paper we  prove the following:

\begin{theorem}\label{points of type S intro} For every $a>0$, $\varphi^{-1}(S_a(X))\neq\emptyset$ if and only if $\varphi^{-1}(S_a(X))$ is full for the Lebesgue measure on $M$. 
\end{theorem}

Which means that, as soon as the image of $\varphi$ touches $S_a(X)$, it is almost totally contained in it, and consequently it contains few rational points. 

Another interesting consequence of Theorem \ref{points of type S intro} is that, as soon as $\varphi^{-1}(S_a(X))\neq\emptyset$, {\it every relatively compact open subset of $M$ is of type $S_a$ for some $a$}. What we actually prove is a bit more general then Theorem \ref{points of type S intro}: in order to obtain the fullness of the set $\varphi^{-1}(S_a(X))$, it suffices the existence of a sufficiently small subset of $M$ of type $S_a$ with respect to $\varphi$. Thus, an interesting corollary of Theorem \ref{points of type S intro} is: 

\begin{theorem}\label{foliations and S points intro} Let $\cF$ be a foliation on a smooth quasi projective variety $Z$ defined over a number field $K$. Let $p\in Z(K)$ be a rational point and $h:\Delta_1\to Z_\sigma(\C)$ be the analytic leaf  of $\cF$ through $p$. Then, if $h(\Delta_1)$ is Zariski dense in $Z$, we have that $h^{-1}(S(Z_K))$ is full in $\Delta_1$.\end{theorem}

The proof of Theorem \ref{points of type S intro} is quite long and relies on the standard Borel--Cantelli Lemma \ref{borelcantelli}, and an estimate of the area of the set where the norm of a section of a hermitian line bundle on a Riemann surface is small. This estimate relies on some tools in compex analysis and the classical Bloch--Cartan estimate (cf. the beginning of sub section \ref{section small sections}). 

In the last section \ref{conclusion} we address some questions which arise naturally and that, we hope, will be clarified in a future.

\section{Notations and basic facts from arithmetic  geometry, complex analysis and measure theory.}



\

\subsection{Tools and notations from arithmetic geometry and Arakelov theory} 

Let $K$ be a number field and $O_K$ be its ring of integers. We will denote by $M_K^\infty$ the set of infinite places of $K$. We fix a place $\sigma_0\in M_K^\infty$. 

Let $X_K$ be a projective variety of dimension $N$ defined over $K$

If $\tau\in M_K^\infty$  and $F$ is an object over $X_K$ ($F$ may be a sheaf, a divisor, a cycle...), we will denote by $X_\tau$ the complex variety $X_K\otimes_\tau\C$ and by $F_\tau$ the restriction of $F$ to $X_\tau$. 

A model $\cX\to\Spec(O_K)$ of $X_K$ is a  flat projective $O_K$ scheme whose generic fiber is isomorphic to $X_K$. Suppose that $L_K$  and $\cX$ are  respectively a line bundle over $X_K$ and a model of it; We will say that a line bundle $\cL$ over $\cX$ is a model of $L_K$ if its restriction to the generic fiber is isomorphic to $L_K$. 

If $\cX$ is a model of $X_K$, an hermitian line bundle $\overline\cL=(\cL, \Vert\cdot\Vert_\sigma)_{\sigma\in M_K^\infty}$ is a line bundle over it equipped, or every $\tau\in M_K^\infty$ a metric on $L_\tau$ with the condition that, if $\sigma=\overline\tau$ then the metric over  $L_\sigma$ is the conjugate of the metric on $L_\tau$.

If $X_K$ is a projective variety, it is easy to see that for every line  bundle $L_K$  on $X_K$, we can find an embedding $\iota : X_K\hookrightarrow P_K$,  where $P_K$ is a smooth projective variety and $L=\iota^\ast(M)$ with $M$ line bundle on $P_K$. A metric on $K_L$ will be said to be smooth if it is the restriction of a smooth metric on $M$.

Let $\overline\cL=(\cL, \Vert\cdot\Vert_\sigma)_{\sigma\in M_K^\infty}$ be an hermitian line bundle on a model $\cX$  of $X_K$. If $s\in H^0(X_K, L^d_K)$ is a non zero section, we will denote by $\log^+\Vert s\Vert$ the real number $\sup_{\tau\in M_K^\infty}\{ 0,\log\Vert s_\tau\Vert_\tau\}$. More generally, of $a$ is a real number,  we will denote by $a^+$ the real number $\sup\{a,0\}$ and by $a_+$ the real number $\sup\{ 1, a\}$.

\begin{definition} An {\rm arithmetic polarization $(\cX,\overline\cL)$ of $X_K$} is the choice of the following data:

\itemize
\item An ample line bundle $L_K$ over $X_K$
\item A projective model $\cX\to\Spec(O_K)$ of $X_K$ over $O_K$.
\item A relatively ample line bundle hermitian line bundle $\overline\cL$ over $\cX$ which is a model of $L_K$.
\item For every $\tau\in M_K^\infty$ we suppose that the metric on $L_\tau$ is smooth and positive.

\end{definition}
\enditemize

We recall the following standard facts of Arakelov theory:

-- If $L$ is an hermitian line bundle over $\Spec(O_K)$ and $s\in L$ is a non vanishing section, we define
\begin{equation}
\widehat{\deg}(L):=\log(Card(L/sO_k))-\sum_{\sigma_in M^\infty_K}\log\Vert s\Vert.
\end{equation}
If $E$ is an hermitian vector bundle of rank $r$ on $\Spec(O_K)$, we define $\widehat{\deg}(E):=\widehat{\deg}(\wedge^{r} E)$ and the slope of $E$ is $\widehat{\mu}(E)={{\widehat{\deg}(E)}\over{r}}$.  

-- Within all the sub bundles of $E$ there is one whose slope is maximal, we denote by $\widehat{\mu}_{\max}(E)$ its slope. It is easy to verify that $\widehat{\mu}_{\max}(E_1\oplus E_2)=\max\{ \widehat{\mu}_{\max}(E_1), \widehat{\mu}_{\max}(E_2)\}$.

 We will need the following version of the Siegel Lemma: 

\begin{lemma}\label{siegel} (Siegel Lemma) Let $E_1$ and $E_2$ be hermitian vector bundles over $O_K$. Let $f:E\to E_2$ be a non injective linear map. Denote by $m=rk(E_1)$ and $n=rk(Ker(f))$. Suppose that there exists a positive real constant $C$ such that:

a) $E_1$ is generated by elements of $\sup$ norm less or equal than $C$.

b) For every infinite place $\sigma$ we have $\Vert f\Vert_\sigma\leq C$ 

Then there exists an non zero element $v\in Ker(f)$ such that
\begin{equation}
\sup_{\sigma\in M^\infty_K}\{ \log\Vert v\Vert_\sigma\}\leq{{m}\over{n}}\log(C^2)+\left( {{m}\over{n}}-1\right) \widehat{\mu}_{\max}(E_2)+3\log(n)+A
\end{equation}
where $A$ is a constant depending only on $K$. 
\end{lemma}

A proof of this version of Siegel Lemma can be found in \cite{gasbarri}.

-- Let $L$ be an hermitian ample line bundle on a projective variety $Z$ equipped with a smooth metric $\omega$. We suppose that the metric on $L$ is smooth. Over $H^0(Z, L^d)$ we can define two natural norms:
 \begin{equation}
 \Vert s\Vert_{\sup}:=\sup_{z\in Z}\Vert s\Vert (z)\}\;\; {\rm and}\;\;\Vert s\Vert_{L^2}:=\sqrt{\int_Z\Vert s\Vert^2 \omega^n}.
 \end{equation}
These norms are comparable: we can find constants $C_i$ such that

\begin{equation}\label{gromov}
C_1\Vert s\Vert_{L^2}\leq \Vert s\Vert_{\sup}\leq C_2^d\Vert s\Vert_{L^2}.
\end{equation}

This statement (due to Gromov) is proved for instance in \cite{SABK} Lemma 2 p. 166 when $Z$ is smooth. The general statement can be deduced by taking a resolution of singularities (remark that the proof of \cite{SABK} Lemma 2 p. 166  do not require that $L$ is ample). 

-- If $(\cX,\cL)$ is an arithmetic polarization of $X_K$, then we we can find constants  $C_1$ and $C_2$  such that
\begin{equation}\label{eq:counting1}C_1^{d^{N+1}} T^{d^N}\leq \Card\left(\{ s\in H^0(\cX, \cL^d)\;\; /\sup_{\tau\in M_K^{\infty}}\{\Vert s\Vert_\tau\}\leq T\}\right)\leq C_2^{d^{N+1}} T^{d^N}.\end{equation}
This is a consequence of \cite{Zhang}, Theorem 1.4 , \cite{GILLETSOULEISTRAEL} Theorem 2 and the comparaison above. 

-- If $\cL$ is an arithmetically ample line bundle, then for $d$ sufficiently big, the lattice $H^0(\cX,\cL)$ is generated by sections of $\sup$ norm less or equal than one. Cf. \cite{Zhang} for a proof. 

-- Let $L/K$ be a finite extension and $O_L$ the ring of integers of $L$. An $L$--point of $X_K$ is a $K$--morphism $P_L:\Spec(L)\to X_K$. The set of $L$ points of $X_K$ is noted  $X_K(L)$. If $(\cX,\cL)$ is an arithmetic polarization of $X_K$, by the valuative criterion of properness, every $L$--point $P_L:\Spec(L)\to X_K$ extends uniquely to a $O_K$--morphism $P_{O_L}:\Spec(O_L)\to \cX$. In this case, $P_{O_L}^\ast(\cL)$ is an hermitian line bundle on $\Spec(O_L)$. We define the height of $P_L$ with respect to $\cL$ to be the real number $h_\cL(P_L):={{\widehat{\deg}(P_L^\ast(\cL))}\over{[L:\Q]}}$.

-- {\it Liouville inequality}:  Let $p\in X_K(K)$ be a rational point. Let $p_0\in X_{\sigma_0}$ be its image. Then, for every positive integer $d$ and  global section $s\in H^0(\cX, \cL^d)$ such that $s(p)\neq 0$ we have
\begin{equation}\label{liouville1} 
\log\Vert s\Vert_{\sigma_0}(p)\geq-[K:\Q]h_\cL(p)\cdot(\log\Vert s\Vert+d).
\end{equation}
For a proof of this form of Liouville inequality cf. \cite{gasbarri2} Theorem 3.1.

\subsection{Tools and basic facts from Nevanlinna theory}.   We will denote by $\Delta_r$ the disk $\{ |z|<r\}$. Let $M$ be a non compact Riemann surface and $U\subset M$ be a relatively compact open set whose border is a smooth Jordan curve which we denote by $\partial U$.  We denote by $\overline U$ the closure of $U$ in $M$.

Denote by $d$ the standard differential operator on functions on $M$ and write it as $d=\partial+\overline{\partial}$.  We denote by $d^c$ the operator $d^c:={{1}\over{4\pi i}}\cdot (\partial-\overline{\partial})$. The operator $dd^c: C^\infty(M)\to A^{(1,1)}(M)$ is called the Laplace operator. 
A function $f$ such that $dd^c(f)=0$ is said to be harmonic. 

The Laplace operator can be extended to an operator from the vector space generated by smooth and sub harmonic functions. 

-- The Green function on $U$ is a function $g_U(z,w): U\times U\setminus D\to [0,+\infty[$ (here $D$ is the diagonal) such that:

(i) For every $p\in U$ the function $g_U(z;\cdot)$ is $\cC^\infty$ in $U\setminus\{p\}$ and harmonic there. Moreover $g_u(p,\cdot)|_{\partial U}=0$.

(ii) if $\iota:\Delta_1\to U$ is a holomorphic embedding such that $\iota(0)=p$, then $\iota^\ast(g(z,\cdot))+\ln|z|$ extends to an harmonic function on $\Delta_1$.

-- Fix a point $z_0\in U$.  If we extend by zero outside $\overline U$ the Green function $g_U(z_0,z)$, the following equation holds: $dd^cg_U(z_0, z)=\delta_{z_0}-d\mu_{U,z_0}$; where $d\mu_{U,z_0}$ is a measure of total mass one supported on $\partial U$ ($\delta_{z_0}$ being the Dirac measure with support on $z_0$). 

Suppose that $L$ is a hermitian line bundle on $M$. Suppose that $s$ is a meromorphic section of $L$. 

-- If $s$ is a meromorphic section of $L$; we denote by $div(s)$ the formal sum $div(s):=\sum_{z\in M}v_z(s)[z]$; where $v_z(s)$ is the multiplicity of $s$ in $z$. This sum may be infinite but  its restriction to every relatively compact open set is finite. 

-- For every $z_0\in M$ denote by $\Omega_{z_0}M$ the cotangent space of $M$ in $z_0$. Let $s\in H^0(M,L)$ be a non vanishing section. If $v_{z_0}(s)=n$ then $s$ defines an element $j^n(s)_{z_0}\in L_{z_0}\otimes \Omega_{z}^{\otimes n}$ called the {\it $n$--th jet of $s$ at $z_0$}. 

-- The current $\delta_{\div(s)}-dd^c\ln\Vert s\Vert$ ($\delta_D$ being the dirac measure with support  on the divisor $D$) extends to a $\cC^\infty$ $(1,1)$--form $c_1(L)$ called {\it the first Chern Class of $L$}. 

-- The function
\begin{equation}
T(z,L,r):=\int_0^r{{dt}\over{t}}\cdot\int_{g_U(z;\cdot)<\ln(t)}c_1(L)
\end{equation} 
is called {\it the Nevanlinna Characteristic function} of $L$ with respect to $U$ and base point $z$. It is linear as  a function of $L$.

-- The Nevanlinna First Main Theorem holds: if $s$ is a meromorphic section of $L$ such that  $v_z(s)=0$ then
\begin{equation}
T(z,L,r)+\int_{\partial U}\ln\Vert s\Vert d\mu_{U,z}=\sum_{w\in U}v_w(s)\cdot g_U(z,w)+\ln\Vert s\Vert(z).
\end{equation}
More generally, if we fix an hermitian metric on $T_{z}M$, $s\in H^0(M, L)\setminus\{0\}$ and $v_z(s)=n$, then
\begin{equation}
T(z,L,r)+\int_{\partial U}\ln\Vert s\Vert d\mu_{U,z}=\sum_{w\in U}v_w(s)\cdot g_U(z,w)+\ln\Vert j^n(s)_z\Vert+n\cdot C.
\end{equation}
Where $C$ is a constant depending only on the metric on $T_zM$.

\subsection{Tools from complex analysis and measure theory} Let $X$ be  a smooth projective variety  of dimension $N$ defined over $\C$. A positive $(1,1)$ form $\omega$ on $X$ induces a volume form $\omega^N$ and consequently a measure $\mu_\omega(\cdot)$ on $X$.  Let  $A\subset X$ be a subset. We will say that {\it $A$ is full in $X$} if the measure of $\mu_\omega(X\setminus A)=0$.  If $\omega_1$ is another positive $(1,1)$ form on $X$, then by compactness of $X$ it is easy to see that $\mu_\omega(X\setminus A)=0$ if and only if $\mu_{\omega_1}(X\setminus A)=0$; thus the "fullness" of $A$ is independent on the chosen metric. 

We recall the classical Theorem of Borel--Cantelli, which can be found in any standard book in measure theory:

\begin{proposition} \label{borelcantelli} Let $X$ be a variety equipped with the Lebesgue measure $\mu$. Let $\{A_n\}_{n\in \bN}$ be a sequence of measurable sets of $X$ such that 
$$\sum_{n=1}^\infty\mu(A_n)<\infty$$
then
$$\mu(\bigcap_{n=1}^\infty\bigcup_{k\geq n}A_k)=0.$$
That means that almost all $x\in X$ belong only to finitely many $A_n$. 
\end{proposition}

\

\section{Degree of a divisor on a bounded domain.}\label{degree of a divisor}

\

In this section we will work over $\C$. 

Let $(X, L)$ be a projective variety with an ample line bundle equipped with a positive metric. 

Let $M$ be a Riemann surface and $U\subset M$ be a relatively compact open domain. Let $\varphi:M\to X$. 

Fix a positive integer $d$.  If $s\in H^0(X,L^d)\setminus\{0\}$, we may write $\div(\varphi^\ast(s))=\sum_{z\in M}n_z(s)\cdot z$ where $n_z(s)$ is the multiplicity of $\varphi^\ast(s)$ at $z$ and it is a positive integral number which is zero for every $z$ up to a (at most) countable set. 

More generally, we write $\div_U(\varphi^\ast(s))=\sum_{z\in U}n_z(s)\cdot z$. This is a finite sum because $U$ is relatively compact. 

We will denote by $\deg_U(s)$ the positive integer $\deg(\div_U(\varphi^\ast(s))=\sum_{z\in U}n_z(s)$. Observe that, again,  this degree is finite because $U$ is relatively compact. It is the number of zeros of $\varphi^\ast(s)$  restricted to $U$ counted with multiplicities. 

Let $W\subset U$ be a compact set. 

Let $s\in H^0(X,L^d)$ be a global section and suppose that $\varphi^\ast(s)|_T$ does not vanish identically. 

Denote by $\Vert s\Vert_W$ the real number $\sup_{z\in W}\{\Vert \varphi^\ast(s)\Vert(z)\}$. 

The norm of a global section and the number $\deg_U(s)$ are related by the following theorem:

\begin{theorem}\label{norms and zeros 0} Under the hypotheses above we can find constants $A$ and $B$, depending on $L$, $U$, $\varphi$ and $W$ such that
\begin{equation}
A\cdot d + B\cdot(\log\Vert s\Vert_{\sup} - \log\Vert s\Vert_W)\geq \deg_U(s).
\end{equation}
\end{theorem}

\begin{proof} We fix a relatively compact open neighborhood $V$ of $\overline{U}$ (closure of $U$) with smooth border.  


We denote by $g_V(z,w)$ the Green function of $V$. 

Since $g_V(z,w)$ is bigger or equal then zero on the closure of $V$ and it is subharmonic there, by the standard mean inequality for subharmonic functions, one sees that $g_V(z,w)>0$ on $V$ (and $g_V(z,w)=0$ on the border $\partial V$). Consequently there is a constant $a>0$ (depending only on $U$ and $V$) such that
$g_V(z,w)>a$ for every $z$ and $w$ in $U$. 

 \begin{lemma}\label{uniform T} We can find a constant $B_1$, which depends only on $V$, $\varphi$ and $L$, such that, for every $w_0\in V$ we have
 \begin{equation}
  \int_0^1{{dt}\over{t}}\int_{g_V(w_0,\cdot)\leq\log(t)}\varphi^\ast(c_1(L))\leq B_1
  \end{equation}
  for every $w_0\in V$. \end{lemma}

\begin{proof} Since every line bundle is trivial on $V$, we can choose a trivialization $\varphi^\ast(L)=\cO_V\cdot e$. Moreover the norm $\Vert e\Vert(z)$ is bounded on $\overline V$. Fix a constants $A_i$ such that 
 $A_1\leq \Vert e\Vert(z)\leq A_2$ for every $z\in \overline V$. 
 
 Let $w_0\in V$. The first main theorem with base point $w_0$ applied to the section $e$ gives
 \begin{equation}
 \int_0^1{{dt}\over{t}}\int_{g_V(w_0,\cdot)\leq\log(t)}\varphi^\ast(c_1(L))+\int_{\partial V}\log\Vert e\Vert d\mu_{V,w_0}=\log\Vert e\Vert(w_0)
 \end{equation}
 The conclusion of the Lemma follows\end{proof}

  Let $s\in H^0(X,L^d)\setminus\{0\}$. Let $w_1\in W$ be a point such that $\Vert \varphi^\ast(s)\Vert (w_1)=\Vert s\Vert_W$. 
  
  
We now apply the first main theorem to $s$, using as base point $w_1$ and we find:
\begin{equation}
 d\cdot\int_0^1{{dt}\over{t}}\int_{g_V(w_1,\cdot)\leq\log(t)}\varphi^\ast(c_1(L))+\int_{\partial V}\log\Vert \varphi^\ast(s)\Vert d\mu_{V,w_1}\geq \sum_{z\in U}n_z(s)(g_V(w_1,z))+\log\Vert s\Vert_W.
 \end{equation}
 Thus
 \begin{equation}
 A_1\cdot d+\log\Vert s\Vert_{\sup}\geq  \sum_{z\in U}n_z(s)\cdot a+\log\Vert s\Vert_W.
 \end{equation}
 
The conclusion  of the Theorem follows.\end{proof}

 we can deduce from Theorem \ref{norms and zeros 0}  a geometric reformulation of the key Lemma 1 of \cite{vanderpoorten}.

Over the vector space $H^0(X,L^d)$ we have two natural norms: Let $s\in H^0(X,L^d)\setminus\{0\}$:

-- the sup norm on $X$: $\| s\Vert_{\sup}:=\sup\{\Vert s\Vert (x) \;/\; x\in X\}$.

-- The sup on $W$ norm:  $\| s\Vert_T:=\sup\{\Vert \varphi^\ast(s)\Vert (z) \;/\; z\in W\}$.

In general we will have $\| s\Vert_W\leq \Vert s\Vert_{\sup}$ and these two norms may be quite different. 

We will denote by $\|\varphi\|(L,W,d)$ the real number 
\begin{equation}
\|\varphi\|(L,W,d):=\sup\{{{\| s\|_{\sup}}\over{\|s\|_W}}\;/\; s\in H^0(X,L^d)\setminus\{0\}\}.
\end{equation}

Thus we find
\begin{proposition}\label{norms and zeros} Under the hypotheses above, we can find positive numbers  $A$ {\rm depending only on $U$ and $M$} and  $B$ {\rm depending only on $U$, $L$ and $\varphi$} such that, for every section $s\in H^0(X,L^d)\setminus\{0\}$, we have

\begin{equation}
 A\cdot\log\|\varphi\|(L,U,d)+B\cdot d\geq \deg_U(s)
\end{equation}
\end{proposition}

In order to prove the Proposition, it suffices to suppose that $\|s\Vert_{\sup}=1$ and consequently  $\log\Vert\varphi^\ast(s)\Vert(w_1)\geq-\log\|\varphi\|(L,U,d)$.


\

\section{Bombieri--Pila revisited}

\

In the seminal paper \cite{BOMBIERIPILA} the authors proved that, on the graph of a real analytic function, there are at most $O(\exp(\epsilon T))$ points of (logarithmic) height at most $T$. In this section we will prove a similar result but more in the spirit of the geometric transcendental theory. One can remark that the proof given is more in the spirit of classical transcendental theory. 

Let $M$ be a Riemann surface and Let $U\subset M$ be a relatively compact open set.  Let $(X,L)$ be a polarized projective variety defined over a number field $K$. We fix an embedding $K\hookrightarrow \C$. 

We fix a holomorphic map $\varphi: M\to X(\C)$ with Zariski dense image.

For every positive number $T$ let 
\begin{equation}
S_U(T):\left\{ z\in U \;/ \; \varphi(z)\in X(K) \; {\rm and}\; h_L(\varphi(z))\leq T\right\}
\end{equation}
and Let $A_U(T)$ be the cardinality of it. 

In this section we prove the following (slight) generalization of the Bombieri--Pila Theorem in this contest:

\begin{theorem}\label{bombieri pila} For every positive number $\epsilon$, we have
\begin{equation}
A_U(T)\ll \exp(\epsilon T)
\end{equation}
where the involved constants depend on $U$, $\varphi$, $\epsilon$ and $L$ but not on $T$. \end{theorem}

Before we give the proof of the theorem, we prove some lemma which will be useful in the sequel.

We fix a relatively compact open set $V\subset M$ which contains $\overline U$. Let $g_V(z,w)$ be the Green function of $V$ and we fix a smooth metric $\omega$ with induced distance $d_V(z,w)$ on $V$.

\begin{lemma}\label{metrics and green} With the notations above, the function $\log(d_V(z,w))+g_V(z,w)$ extends to a continuous  function on $V\times V$.
\end{lemma}
\begin{proof} Fix a small disk $D$ in $V$ with coordinate $z$. We claim that the function $g_V(z,w) +\log(|z-w|)$ is continuous on $D\times D$.

Indeed, denote by $h_V(z,w):=g_V(z,w) +\log(|z-w|)$. By properties of the Green function, for every $z_0$ and $w_0$ in $D$, the functions $h_V(z_0, w)$ and $h_V(z,w_0)$ are harmonic and bounded on $D$. Fix $z_0\in D$. We have that $|h_V(z,w)-h_V(z_0,z_0)|\leq |h_V(z,w)-h_V(z_0,w)|+|h_V(z_0,w)-h_V(z_0,w_0)|$. By Harnack's inequality, the two terms of this sum are bounded by $\epsilon(h_V(z_0,w_0)+A)$ as soon as $(z,w)$ is sufficiently near to $(z_0,w_0)$. The claim follows. 

On the other side, we claim that, on the same disk, also the function ${{d_V(z;w)}\over{|z-w|}}$ extends to a positive continuous function on $D\times D$.

We may suppose that, on $D$, the hermitian metric is give by a $(1,1)$ form  $\omega= i F(z)dz\wedge d\overline z$ where $F(z)$ is a positive smooth function on $D$. 

In order to prove the claim, we will prove that, given $z_0\in D$, then $\lim_{(z;w)\to (z_0.z_0)} {{d_V(z;w)}\over{|z-w|}}=F(z_0)$. Fix $\epsilon>0$.  We can choose two concentric disks $U_\epsilon\subset V_\epsilon$ centered in $z_0$ such that, for every $z\in U_\epsilon$ we have
$F(z_0)-\epsilon\leq F(z)\leq F(z_0)+\epsilon$ and, for every couple $z$ and $w$ in $V_\epsilon$ the geodesic curve between $z$ and $w$ with respect to the metric $\omega$ is entirely contained in $V_\epsilon$ (in order to obtain this, it suffices that, for every $z\in U_\epsilon$,  the set of points with $\omega$--distance less or equal to $2\epsilon$ from $z$ is entirely contained in $V_\epsilon$). 

Let $z$ and $w$ in $U_\epsilon$. We denote by $\alpha_\omega(z,w)$  and $\alpha_s(z,w)$  the geodesic paths between $z$ and $w$ with respect to the metric $\omega$ and the standard euclidean metric on the disk respectively. If $\beta$ is a path, we denote by $\ell_\omega(\beta)$ and $\ell_s(\beta)$ the length of it with respect to the metric $\omega$ and the standard euclidean metric.

We recall that the distance between $z$ and $w$ may be defined in two ways: either it is the length of the geodesic path between $z$ and $w$ or it is the minimum between the lengths of all the paths between them. 

If $z$ and $w$ are points in $U_\epsilon$ we have that $d_V(z,w)\leq \ell_\omega(\alpha_s(z,w))(\leq F(z_0)+\epsilon)|z-w|$ and $(F(z_0)-\epsilon)|z-w|\leq (F(z_0)-\epsilon)\ell_s(\alpha_\omega(z,w))\leq \ell_\omega(\alpha_\omega(z,w))=d_V(z,w)$. Since $\epsilon$ is arbitrarily small, the claim follows. Thus $\log(d_V(z,w))-\log(|z-w|)$ extends to a continuos function on $D$.

Consequently $g_V(z,w)+log(d_V(z,w))$, being the sum of two continuous functions is continuous.\end{proof}

As a corollary of the Lemma above, we find that there are  constants $A_i$,  depending only on $U$, $V$ and the chosen metric, such that, for every $(z,w)\in U\times U$, we have
\begin{equation}\label{green and distance}
A_1\leq \log(d_V(z,w))+g_V(z,w) \leq A_2
\end{equation}

The main lemma we need is the following:
\begin{lemma}\label{distance lemma} Let $d_0$ a sufficiently big integer. There exist constants $C_i$, depending only on $U$ and $d_0$ depending only on $U$, $X$ and $\varphi$, with the following property: let $d$ be a positive integer, $W$ be a open set in $U$ and $r_W$ be its diameter. If $d\geq d_0$ and $r_W\leq C_1\exp(-{{C_2T}\over{d^{n-1}}})$, then there exists a section $s\in H^0(X,L^d)\setminus\{ 0\}$ such that $\varphi^\ast(s)$ vanishes on every point of $S_W(T)$.
\end{lemma}
\begin{proof}  Denote by $h^0(X,L^d)$ the rank of $H^0(\cX, \cL^d)$, it is well known that  can find a positive constant $B_1$ such that, for every $\epsilon_1>0$ and $d$ sufficiently big, we have 
\begin{equation}
B_1(1-\epsilon_1)d^n\leq h^0(X,L^d)\leq B_1(1+\epsilon_1) d^n.
\end{equation}
We also recall that, by Zhang's theorem \cite{Zhang}, we may suppose that $H^0(\cX,\cL^d)$ is generated by sections of norm less or equal than one. 

Fix an $\epsilon>0$ and we suppose that $d$ is big enough to have that $\epsilon h^0(X,L^d)>1$. 
Choose an integer $A$ such that $(1-2\epsilon)h^0(X,L^d)\leq A\leq (1-\epsilon)h^0(X,L^d)$ and a subset $H_W(d)\subset S_W(T)$ of cardinality $A$ (if the cardinality of $S_W(T)$ is smaller than the Lemma easily follows from linear algebra). 

Denote by $E(T)$ the  $O_K$ module $\oplus_{z\in H(T)}\cL^d|_{f(z)}$. The rank of $E$ is $A$ and $\mu_{\max}(E(T))\leq dT$.

We have a natural restriction map
\begin{equation}
\delta_T: H^0(\cX,\cL^d)\longrightarrow E(T).
\end{equation}

By Gromov theorem \ref{gromov}, if we put on $H^0(\cX, \cL^d)$ the $L_2$ hermitian structure and on $E(T)$ the direct sum hermitian structure, the norm of $\delta$ is bounded by $C_0d$ for a suitable constant $C_0$. 

Denote by $K(T)$ the kernel of $\delta_T$ and by $k(T)$ its rank. By construction we have that 
\begin{equation}
{{h^0(X,L^d)}\over{k(T)}}\leq{ {h^0(X,L^d)}\over {h^0(X,L^d)-(1-\epsilon)h^0(X,L^d)}}={{1}\over{\epsilon}}.
\end{equation}
We may then apply Siegel Lemma \ref{siegel} and we obtain that there is a non vanishing section $s\in H^0(\cX, \cL^d)$ such that $\varphi^\ast(s)$ vanishes on every point of $E(T)$ and $\log\Vert s\Vert_{\sup}\leq C_1dT$ for a suitable constant $C_1$ independent on $T$. 

Let $w\in S_W(T)$ we will now prove that $\varphi^\ast(s)$ vanishes also on $w$.  Suppose that $\varphi^\ast (s)$ do not vanish on $w$. 

We apply Nevanlinna First Main Theorem to the section $\varphi^\ast(s)$ over the open set $V$ with base point $w$ and obtain a constant  $A_2$ independent on $w$ such that
\begin{equation}
dA_2+\int_{\partial V}\log\Vert \varphi^\ast(s)\Vert d\mu_{V,w}\geq \sum_{z\in S_W(T)}g_w(w,z)+\log\Vert\varphi^\ast(s)\Vert(w).
\end{equation}
Remark that $A_2$ is independent on $w_0$ because of Lemma \ref{uniform T}. From the upper bound of the norm of $\varphi^\ast(s)$, the estimate \ref{green and distance}, the hypothesis on the diameter of $W$ and Liouville inequality \ref{liouville1} we obtain
\begin{equation}
dA_1+C_1dT\geq A(A_1-\log(r_W))-dT-A_3
\end{equation}
for suitable constants $A_i$ independent on $w$ (and on $W$).
Since $A\geq B_1\cdot d^n$ (for a suitable constant depending only on $L$ and $d$), we obtain, constants $B_i>0$ (independent on $w$) such that
\begin{equation}
\log(r_w)\geq B_2-B_3{{T}\over{d^{n-1}}}.
\end{equation}
And this contradicts the hypothesis of the Lemma as soon as $C_i>B_i$. \end{proof} 

\begin{proof} ({\it of Theorem \ref{bombieri pila}}) Now the proof is immediate. We fix the constants $C_i$ as in the Lemma \ref{distance lemma} and $d$ a sufficiently big integer such that ${{C_2}\over{d^{n-1}}}<\epsilon$. 

For every $T$, we may write $U=\bigcup W_i$, were $W_i$ are open set of diameter at most 
$C_1\exp(-{{C_2}\over{d^{n-1}}}T)$. 

The number of these $W_i$ is at most $2C_1\exp({{2C_2}\over{d^{n-1}}}T)$. 

By Lemma \ref{distance lemma}, for every $W$ there is a non vanishing section $s_W\in H^0(X;L^d)$, such that $S_W(T)\subset\{ z\in U\;/\; \varphi^\ast(s_W)=0\}$.  By Theorem \ref{norms and zeros}, for each $s_W$, the cardinality of  the set $\{ z\in U\;/\; \varphi^\ast(s_W)=0\}$ is bounded by a constant (depending on $d$). Thus the conclusion of the Theorem follows. \end{proof}

We would like to remark that, strictly speaking, one could give a proof of Theorem \ref{bombieri pila} simply by reduction to the case when $U$ is a unit disk (with a a coordinate $z$). This would allow, essentially to avoid Lemma \ref{metrics and green} (or more specifically, use an explicit  and much simpler version of it on a disk). Never the less, the interest of this proof is  to point out the main analytic tool which is needed which is Lemma \ref{metrics and green}. 

\subsection{On distribution of rational points on a Riemann Surface} Examples provided by Surroca \cite{surroca}, show that that the bound of Bombieri and Pila, Theorem \ref{bombieri pila}, cannot be improved. Never the less, given an analytic map as above, we may ask for which value of $T$ the number $A_U(T)$ is big and for which it is not. We will show now that, for many values of $T$, the number $A_U(T)$ is actually bounded by a polynomial in $T$. In order to quantify this we give the following definitions: 

\begin{definition}  Let $\varphi :M\to X$, and $U$ as in Theorem \ref{bombieri pila}. Let $\gamma\in \bR>0$ and $\epsilon>0$.  denote by $L(\varphi,\gamma,\epsilon )\subset \bR$ the following set:
\begin{equation}
L(\varphi,\gamma, \epsilon):=\left\{ T\in \bR\; /\; A_U(T)\leq \epsilon T^\gamma\right\}.
\end{equation}
\end{definition}

In this section we will prove that, even if the Bombieri Pila estimate is optimal, for big values of $\gamma$ the set $L(\varphi,\gamma, \epsilon)$ is very big. Actually  it is almost as big as we want. 

\begin{definition} Let $A>1$ be a real number. We will say that an interval $I\subset \bR$ is {\rm geometrically wider than $A$} if there is $t\in I\cap ]1,+\infty[$ such that the interval $[t,At]$ is contained in $I$.\end{definition}

\begin{theorem}\label{rare distribution 1} Let $\varphi :M\to X$, and $U$ as in Theorem \ref{bombieri pila}. Let $\gamma> {{n}\over{n-1}}$, $A>1$ and $\epsilon>0$.  
Then the set  $L(\varphi,\gamma, \epsilon)$ is unbounded and contains infinitely many  intervals  $I\subset \bR$ which are geometrically wider than $A$. 
 \end{theorem}
 
 This theorem generalize Theorem 1.3 of \cite{surroca}. The Theorem means that there are infinitely intervals of the form $[r, Ar]$ with $\sup\{r\}=+\infty$ and such that, for every $T$ in one of these intervals,  the number $A_U(T)$ is "small". 
 
 \begin{proof} As in the proof of Theorem \ref{bombieri pila}, we fix a relatively compact open set $V$ containing $\overline{U}$. We use notations of the proof of Lemma \ref{distance lemma}. 
 
 We suppose that the conclusion of the the Theorem is false. Consequently we may find a strictly increasing sequence $(T_n)_{n\in \bN}$ such that:
 
 -- For every $n$ we have $T_{n+1}\leq A\cdot T_n$;
 
 -- $A_U(T_n)\geq \epsilon T_n^\gamma$;
 
 -- $\lim_{n\to\infty} T_n=+\infty$.
 
 We may also suppose that $T_1$ is very big. In particular we suppose that $(\epsilon T_1)^{\gamma/n}$ is very big. 
 
 Fix an integer $d_1$ such that $(\epsilon T_1)^{\gamma/n}\leq d_1\leq (\epsilon T_1)^{\gamma/n}+1$ and a subset $H_U(T_1)\subseteq S_U(T_1)$ of cardinality $A_{1}$, where $A_1$ is  an integer which verify $(1-2\epsilon)h^0(X,L^d)\leq A_1\leq (1-\epsilon)h^0(X,L^d)$.

 Following the same strategy in the proof of Lemma \ref{distance lemma}, an application of Siegel Lemma \ref{siegel} allows to construct a non vanishing global section $s\in H^0(\cX, \cL^d)$ such that:
 
 -- There exists a constant $C_1$ independent on $T_1$, on $d_1$ such that $\log\Vert s\Vert\leq C_1 T_1^{{{\gamma}\over{n}}+1}$;
 
 -- $\varphi^\ast(s)$ vanishes on every point of $H_U(d_1)$. 
 
 Again, as in the proof of Lemma \ref{distance lemma}, we will now prove that the constructed section $s$ actually vanishes on every point of $S_U(T_1)$. 
 
 Let $w_0$ be a point of $S_U(T_1)$ and suppose that $\varphi^\ast(s)$ do not vanish on it. We apply Nevanlinna First Main theorem, Lemma \ref{uniform T} and Liouville inequality to the section $\varphi^\ast(s)$ and the domain $V$ with base point $w_0$ and we obtain the existence of constants $C_i$, independent on $w_0$ and $s$ such that
\begin{equation}
C_2T^{\gamma/n} +C_1T_1^{{{\gamma}\over{n}}+1}\geq \sum_{z\in H_U(d_1)}g_V(w_0,z)+\log\Vert \varphi^\ast(s)\Vert(w_0).
\end{equation}
By Lemma \ref{metrics and green} and the fact that $\overline{U}$ is compact and contained in $V$, we may find a constant $C_3>0$ such that, for every $z\in U$ we have $g_V(z,w_0)\geq C_3$.

We observe that that the  cardinality $H_U(d_1)$ is lower bounded by $C_4d^n$ (with $C_4$ depending only on $X$ and $L$).  We apply Liouville Inequality to $\varphi^\ast(s)$ and $w_0$ and we find that, for $T_0$ sufficiently big, we have:
\begin{equation}
C_5 T_1^{{{\gamma}\over{n}}+1}\geq C_6 T_1^{\gamma}.
\end{equation}
This is impossible for our choice of $\gamma$ and for $T_1$ big enough. Thus $\varphi^\ast(s)$ vanishes on every point of $S_U(T_1)$.

By induction, we may suppose that $\varphi^\ast(s)$ vanishes on every points of $S_U(T_n)$ and we will now prove that it vanishes on every point of $S_U(T_{n+1})$. 

Suppose that $w_{n+1}$ is an element of $S_U(T_{n+1})$ such that $\varphi^\ast(s)(w_{n+1})\neq 0$.

Once again, we apply Nevanlinna First Main Theorem and Liouville inequality to $\varphi^\ast(s)$ and $V$ with base point $w_{n+1}$, using  the estimates above  on the Green functions we obtain the existence of constants $C_i$ independent on $w_{n+1}$:
\begin{equation}
C_2T_1^{\gamma/n} +C_1T_1^{{{\gamma}\over{n}}+1}\geq C_6T_n^\gamma-C_7T_1^{\gamma/n}T_{n+1}-C_8T_1^{{{\gamma}\over{n}}+1}.
\end{equation}
Which gives, since $T_{n+1}\leq A T_n$,
\begin{equation}
C_1T_n^{{{\gamma}\over{n}}+1}\geq C_6T_n^\gamma.
\end{equation}
And this is impossible as soon as $T_n$ is sufficiently big. 

The conclusion of the Theorem follows because $\varphi^\ast(s)$ is a section of an holomorphic line bundle thus it can vanish only on finitely many points of $U$.   \end{proof}
 
 \
 
\section{Riemann surfaces containing subsets of type $S_a$}

\

In the paper \cite{gasbarri2}, inspired by the work of Chudnovski \cite{CHUDNOVSKY}, we introduced a class of points on a projective variety which  verify some inequalities of Liouville type: the points of type $S_a$. Let's recall (and slightly modify) this definion:

Let $(X_K,\cL)$ be an arithmetically polarized projective variety defined over a number field $K$ of dimension at least two. 

\begin{definition}\label{points of type S 1}  Let $z\in X_K(\C)$. We will say that $z$ is of type $S$ (or that $z\in S(\cX)$) if we can find positive constants $a=a(z,\cL)$,  $A=A(z,\cL)$ depending on $z$ ,$\cL$  such that, for every positive integer $d$ and every non zero global section $s\in H^0(\cX, \cL^d)$ we have that
$$
\log\Vert s_{\sigma_0}\Vert_{\sigma_0}(z)\geq -A(\log^+\Vert s\Vert +d)^a.
$$
Moreover we will denote by $S(X_K)$ the subset of $X_{\sigma_0}(\C)$ of points of type $S$. If $a_0$ is fixed, we will denote by $S_{a_0}(X_K)$ the set of $S$--points of $X$ for which the involved constant $a$ is $a_0$. 
\end{definition}

Observe that $S(X_K)=\cup_{a\geq 0}S_a(X_K)$. 

We proved in \cite{gasbarri2} the following theorem:

\begin{theorem}\label{maingasbarri2} Let $X$ as before. The following properties hold:

a) If $a<\dim(X_K)+1$ then $S_a(X_K)=\emptyset$. 

b) If $a\geq \dim(X_K)+2$, the set $S_a(X_K)$ is full in $X_K(\C)$.

c) Let $Y\subset X(\C)$ be a compact Riemann surface and $a\geq \dim(X_K)+2$, then $Y\cap S_a(X_K)\neq \emptyset$ if and only if $Y\cap S_a(X_K)$ is full in $Y$. 
\end{theorem}

Let $M$ be a Riemann surface and $U\subset M$ be a relatively compact open set. Let $\varphi: M\to X_\sigma(\C)$ be an holomorphic map with Zariski dense image. The interest of points of type $S_a$ is the following theorem, proved in \cite{gasbarri2} Theorem 10.1:

\begin{theorem} \label{theorem 10.1} We keep notations as in the previous section. If, for some real number $a$ we have $\varphi^{-1}(S_a)\cap U\neq \emptyset$, then
\begin{equation}
A_U(T)\ll T^a.
\end{equation}
\end{theorem}

Theorem \ref{theorem 10.1} can be weakened  a bit: instead of supposing that the norm is "big" on a point, we may suppose that it is "big" just on a compact subset of $U$.

\begin{definition}\label{sets of type S} Let $\varphi: M\to X_\sigma(\C)$ and $U$ as before. Let $B\subset U$ be a compact set and $a$ be a real number. We will say that $B$ is a subset of type $S_a$ of $M$  with respect to $\varphi$ if, we can find a positive constant $A>1$ such that, for every positive integer $d$ and $s\in H^0(\cX.\cL^d)\setminus\{0\}$ we have that
\begin{equation}
\log\Vert s\Vert_B\geq -A(\log^+\|s\|+d)^a.
\end{equation}
\end{definition}

It is possible to prove that, in this case, we must have $a\geq \dim(X_K)+1$. 

We can then generalize Theorem \ref{theorem 10.1} to subsets of type $S_a$:

\begin{theorem}\label{poly T} Let $\varphi: M\to X_\sigma(\C)$ and $U$ as before. Suppose that we can find a subset $B\subset U$ of type $S_a$ with respect to $\varphi$. Then 
\begin{equation}
A_U(T)\ll T^{2a}.
\end{equation}
\end{theorem}

\begin{proof} The proof follows the same paths of the proof of Theorem \ref{rare distribution 1} or the proof of Theorem \ref{maingasbarri2}. For reader 's convenience, we give here a short sketch of it.  Fix  $\gamma>\dim(X_K)$. 

Let $T$ be sufficiently big. As in the proof of \ref{rare distribution 1}, we can find an integer $d$ with $T^{\gamma/n}\leq d\leq T^{\gamma/n}+1$ and a global section $s\in H^0(\cX,\cL^d)\setminus\{ 0\}$ such that:

-- $\log\Vert s\Vert_{\sup}\leq C_1T^{\gamma/n+1}$ for a constant $C_1$ independent on $T$;

-- for every $w\in S_U(T)$, we have $\varphi^\ast(s)(w)=0$.

We may now apply Theorem \ref{norms and zeros 0}  and conclude.\end{proof}

Of course, Theorem \ref{theorem 10.1} is just a particular case of Theorem \ref{poly T}. We will now see another application of Theorem \ref{poly T}:

\subsection{Rational points on leaves of one dimensional foliations}

Let $Z$ be a smooth quasi projective variety  of dimension $N$ defined over $K$. Let $T_Z$ its tangent bundle. A one dimensional foliation $\cF$ over $Z$ is a sub line bundle $H\hookrightarrow T_Z$ (the quotient is locally free). 

Fix a one dimensional foliation $\cF$ and a rational point $p\in Z(K)$. Denote by $\widehat{Z}_p$ the formal completion of $Z$ at $p$.

 The {\it formal leaf} of the foliation $\cF$ at $p$ is a formal morphism
$\iota:Spf(K[\![X]\!]):=\widehat{\bA}^1_0\to \widehat{Z}_p$ such that the natural differential morphism 
$\iota^\ast(\widehat{\Omega}^1_{Z_p/K})\to \widehat{\Omega}^1_{\widehat{\bA}^1_0/K}$ factorizes through the natural surjection $\iota^\ast(\widehat{\Omega}^1_{Z_p/K})\to \iota^\ast(H^\wedge)$. Since $K$ is of characteristic zero, the formal leaf exists and it is unique (up to reparametrization). 

By Frobenius theorem over $\C$, we can also find an analytic leaf of the foliation: this is an holomorphic morphism $h:\Delta_1\to Z_\sigma(\C)$ such that $h(0)=p$ and the natural morphism $h^\ast(\Omega^1_Z)\to \Omega^1_{\Delta_1}$ factorizes through the natural surjection $h^\ast(\Omega^1_Z)\to h^\ast(H^\wedge)$. 

By unicity of the formal leaf, the formal completion of the analytic leaf of $\cF$ at $p$ coincides with its formal leaf. 

Denote by $\bA^1_n$ the scheme $\Spec(K[X]/(X^{n+1}))$. It is called the $n$--th formal neighborhood of $0$ in $\widehat{\bA}^1_0$ and there is a  natural inclusion 
$j_n:\bA^1_n\hookrightarrow \widehat{\bA}^1_0$. Denote by $T_0$ the fibre at $0$ of the tangent bundle of $\widehat{\bA}^1_0$, we have a canonical exact sequence
\begin{equation}\label{jets}
0\longrightarrow T_0^{\otimes n}\longrightarrow \cO_{\bA^1_{n+1}}\longrightarrow\cO_{\bA^1_{n}}\longrightarrow 0.
\end{equation}

By construction we have natural inclusions $\iota_n:=\iota\circ j_n:\bA^1_n\to Z$ which factorize through $\iota$. 

Let $\overline{Z}$ be a (smooth) projective compactification of $Z$ and $L$ be a ample line bundle over it. A global section $s\in H^0(\overline{Z},L^d)$ is said to be {\it vanishing at order $n$ over the formal leaf $\iota$ at $p$} if $\iota_{n}^\ast(s)=0$ but $\iota_{n+1}^\ast(s)\neq 0$. In this case, via the exact sequence \ref{jets}, the section  $s$ canonically defines a section $j^n(s)\in T_0^n\otimes L_p$ called the {\it $n$--jet of $s$ at the formal leaf $\iota$}. In this case we will denote the integer $n$ by $ord_{\cF,p}(s)$.

If we fix an arithmetic polarization $(\overline{\cZ}, \cL)$ of the couple $(\overline{Z}, L)$ with $\overline{Z}$ normal. The point $P$ extends to a section $P:\Spec(O_K)\to \overline{\cZ}$. It is possible to fix an integral structure $\cT_0$ of $T_0$ for which the following holds:

\smallskip

-- There exists a constant $C_p$ such that, if $s\in H^0(\overline{\cZ},\cL^d)$ vanishes at the order $n$ over the formal leaf $\iota$, then
$C_p^n\cdot n!\cdot j^n(s)\in \cT_0^{\otimes n}\otimes \cL^d|_P$. 

\smallskip

The proof of this fact is the main topic of section 3 of \cite{gasbarri1}.

If we fix an hermitian structure on the $O_K$--module $\cT_0$, a direct application of Schwartz inequality gives the following:

\begin{proposition}\label{foliated liouville} We can find a constant $C$, independent on $s$ for which the following holds: Let $s\in H^0(\overline{\cZ},\cL^d)$ a section vanishing at order $n$ on the formal leaf $\iota$, then
\begin{equation}
\log\Vert j^n(s)\Vert_\sigma\geq -C(n\log(n) +d+\log^+\Vert s\Vert_{\sup}).
\end{equation}\end{proposition}

As a consequence we find that, in order to find a lower bound for the norm of the jet of section vanishing at $p$, we need to bound the order of vanishing. The order of vanishing can be bounded by  the following proposition:

\begin{proposition}\label{zero lemma}  Suppose that the (formal) leaf at $p$ is Zariski dense. We can find a constant $C_p$ for which the following holds: For every positive integer $d$ and every non vanishing global section $s\in H^0(\overline{Z}, L^d)$ we have
\begin{equation}
ord_{\cF,p}(s)\leq C_pd^{\dim(Z)}.
\end{equation}\end{proposition}

The proof of the proposition leans on the following theorem, which, a priori, is just the case when $Z$ is the projective space:

\begin{proposition} \label{nesterenko1}Let $D:=\sum_{i=1}^NP_i(x_1,\dots, x_N){{\partial}\over{\partial x_i}}$ be a differential operator on the affine space $\bA^N$ with $P_i(x_1,\dots, x_N)\in \C[x_1,\dots, x_N]$. Suppose that the variety $V(P_1,\dots, P_N)$ do not contain the origin. Let $V_{\cF, 0}$ be the formal leaf through the origin of the foliation defined by $D$. Let $\ell$ be the Zariski closure of $V_{\cF,0}$. Then we can find a constant $C$ such that, for every polynomial $Q(x_1,\dots, x_N)\in \C[x_1,\dots, x_N]$ of degree less or equal than $d$ not vanishing identically on $V_{\cF, 0}$, we have
\begin{equation}
ord_{\cF,0}(P)\leq Cd^\ell.
\end{equation}\end{proposition}

The proof of the Proposition above can be found in \cite{Nesterenko} or \cite{Byniamini}. 

In order to deduce Proposition \ref{zero lemma} from Proposition \ref{nesterenko1} it suffices to remark that we can suppose that the involved line bundle $H^{-1}$ of the foliation is very ample and consequently the foliation comes from the restriction of a foliation on $\P^N$. 

We can deduce  from the proposition above the fact that the number of rational points of height less or equal than $T$ on a disk which is the analytic leaf of a foliation, grows polynomially with $T$:

\begin{theorem}\label{foliations} Let $\cF$ be a foliation on a smooth quasi projective variety $Z$ defined over a number field $K$. Let $p\in Z(K)$ be a rational point and $h:\Delta_1\to Z_\sigma(\C)$ be the analytic leaf  of $\cF$ through $p$. Suppose that the dimension of the Zariski closure of the leaf is $\ell>1$. Let $0<r<1$ be a real number, then, for every $\epsilon>0$ we have
\begin{equation}
A_{\Delta_r}(T)\ll_\epsilon T^{2\ell+\epsilon}.
\end{equation}\end{theorem}
\begin{proof} By Theorem \ref{poly T}, it suffices to prove that such a $\Delta_r$ contains a compact subset of type $S_{\ell +\epsilon}$. 

Replacing $Z$ by the Zariski closure of the leaf and taking a resolution of singularities, we may suppose that the leaf is Zariski dense in $Z$. 

Choose a small disk $\Delta_{r_0}\subset \Delta_r$ (with $r_0<r$). Let $d>0$ be a positive integer and $s\in H^0(\overline{\cZ}, \cL^d)\setminus\{ 0\}$.

Let $n=ord_{\cF, p}(s)$. By Nevanlinna First Main Theorem applied to the disk $\Delta_{r_0}$ and the line bundle $\cL$ we may find a constant $A$ for which
\begin{equation}
A\cdot d +\int_0^{2\pi} \log\Vert s\Vert_\sigma(r_0\cdot\exp(2\pi i \theta){{d\theta}\over{2\pi}}\geq \log\Vert j^n(s)\Vert_\sigma.
\end{equation}
 
Thus there is a point $w_0\in \partial \Delta_{r_0}$ such that
\begin{equation}
\log\Vert s\Vert(w_0)\geq \log\Vert j^n(s)\Vert_\sigma-A\cdot d.
\end{equation}
The conclusion follows from Proposition \ref{zero lemma} and Proposition \ref{foliated liouville}.\end{proof}

Particular cases of Theorem \ref{foliations} have been proved in \cite{Byniamini0} and \cite{comte yomdin}.



\

\section{Riemann surfaces containing points of type $S_a$}

\

We saw in the previous section that if a Zariski dense Riemann surface in a projective variety contains a point of type $S$ then it must contain "few" rational points. We also remarked that, in order to contain few rational points,  it suffices that it contains a subset of type $S$. 

In this section we will show that if a Riemann surface contains a subset of type $S$ which is "sufficiently small", then the set of points of type $S$ contained in it is full for the Lebesgue measure. 

As a corollary, as in the case of compact Riemann Surfaces, (cf. Theorem \ref{maingasbarri2} (c)),  we find that the fact that the intersection of the Riemann surface with points of type $S$ is non empty is equivalent with the fact that the points of type $S$ in it form a full set. 

We fix a Riemann surface $M$ and an holomorphic map $\varphi :M\to X_\sigma(\C)$ with Zariski dense image. Essentially, in this  section we will prove the  following general principle: 

\smallskip

Either the image of $M$ via $\varphi$ do not intersect the set of points  of type $S$ of $X_K$; in this case it {\it might} contain many rational points, or it does intersect the set of points of type $S$; in this case, essentially every point of the image is of type $S$ and there are very few rational points contained  in it. 

\smallskip

A  non compact Riemann surface $M$ may be equipped with a canonical  metric  of constant curvature (up to a scalar factor). Let $\alpha:\tilde{M}\to M$ be the universal covering of $M$. The Riemann surface $\tilde{M}$ is either $\C$ or the disk $\Delta_1$. 

The main theorem of this section is the following:

\begin{theorem}\label{S sets} Suppose that $\varphi:M\to X_\sigma(\C)$ is a holomorphic map as above. Let $B\subset M$ be a compact subset of type $S$ with respect to $\varphi$. Then

-- If $\tilde{M}=\C$, the set $\varphi^{-1}(S(X_K))$ is full in $M$.

-- If $\tilde{M}=\Delta_1$, then there exists a real number $\delta>0$ such that, if the diameter of $B$ is less  than $\delta$, the set $\varphi^{-1}(S(X_K))$ is full in $M$.\end{theorem}

The number $\delta$ can be taken to be $\ln(2)$.

As a corollary we find: 

\begin{theorem}\label{intersecting S points} With the notations as above, then $\varphi^{-1}(S(X_K))\neq \emptyset$ if and only if $\varphi^{-1}(S(X_K))$ is full in $M$ for the Lebesgue measure.
\end{theorem}

Another important corollary of Theorem \ref{S sets} and of the proof of Theorem \ref{foliations} is the following:

\begin{theorem}\label{foliations and S points} Let $\cF$ be a foliation on a smooth quasi projective variety $Z$ defined over a number field $K$. Let $p\in Z(K)$ be a rational point and $h:\Delta_1\to Z_\sigma(\C)$ be the analytic leaf  of $\cF$ through $p$. Then, if $h(\Delta_1)$ is Zariski dense in $Z$, we have that $h^{-1}(S(Z_K))$ is full in $\Delta_1$.\end{theorem}

For the proof, it suffices to remark that, in the proof of Theorem \ref{foliations} we proved that every neighborhood of $p$ in $\Delta_1$ is of type $S$ with respect to $h$.

\subsection{Area of the set of points where a global section is small}\label{section small sections}

Theorem \ref{S sets} will be consequence of the following estimate, which will is interesting in its own.

Let $M$ be a Riemann Surface. We fix an hermitian line bundle $L$ on $M$.  We fix a point $z_0\in U$. We also fix a positive metric on $M$. For every positive integer $d$, we will denote by $H^0_B(M;L^d)$ the subspace of  section $s\in H^0(M, L^d)$ for which $\sup_{z\in M}\{\Vert s\Vert\}\leq C(s)$, where $C(s)$ is a positive constant, depending on $s$. For $s\in H^0_B(M;L^d)$, we will denote by $\Vert s\Vert_M$ the number $\sup_{z\in M}\{\Vert s\Vert\}$. More generally, if $B\subset M$, we will denote by $\Vert s\Vert_B$ the number $\sup_{z\in B}\{\Vert s\Vert\}$. We fix a universal covering $\alpha:\tilde M\to M$ and consequently we can fix a metric $\mu(\cdot)$ with constant curvature on $M$. The distance defined by the metric will be denoted by $d_M(\cdot,\cdot)$.

Let $U\subset M$ be a non dense open set. We will say that $U$ is {\it small} if:

-- when $\tilde M=\C$, the closure $\overline U$ of $U$ is compact;

-- when $\tilde M=\Delta_1$, the closure $\overline U$ of $U$ is compact and its diameter is strictly less than $\delta:=\ln(2)$.

The main theorem of this sub section is:

\begin{theorem}\label{main estimate} Let $M$ be a Riemann surface, $L$ be an hermitian line bundle over it, $U$ be a small open subset of $M$ and $W\subset U$any subset. We can find positive real constants $C_i$, depending only on $M$, $L$  $U$ and $W$ for which the following holds: Let $0<\eta<1$ be a positive constant. For every positive integer $d$ and global section $s\in H_B^0(M,L^d)\setminus\{0\}$ such that $s(z_0)\neq 0$, define
\begin{equation}
B(s):=\left\{ z\in U\;/\; \ln\Vert s\Vert(z)\leq -C_1(\ln\left({{1}\over{\eta}}\right)+1)(\log^+\Vert s\Vert_M+d) +3 \ln\Vert s\Vert_W\right\}.
\end{equation}
Then 
\begin{equation}
\mu(B(s))\leq C_2 \eta^2.
\end{equation}
\end{theorem}

In order to prove Theorem \ref{main estimate}, we need some classical results from complex analysis. The first one is the classical estimate by Bloch and Cartan (\cite{Lang} Theorem 3.1 page 236):

\begin{proposition}\label{bloch cartan} Let $a_1,\dots, a_n$ be $n$ complex numbers (which may not be distinct). Let $H$ be a positive real number. Then the numbers $z\in\C$ for which one has the inequality
\begin{equation}
\prod_{i=1}^n|z-a_i|\leq \left( {{H}\over{2e}}\right)^n
\end{equation}
is contained in the union of at most $n$ discs, such that the sum of the radii is bounded by $H$.
\end{proposition}

We will use this proposition in the following form:

\begin{proposition}\label{bloch cartan2} Let $a_1,\dots, a_n$ and $H$ as in Proposition \ref{bloch cartan}. Then 
\begin{equation}
\mu(\{ z\in\C\;/\; \prod_{i=1}^n|z-a_i|\leq \left( {{H}\over{2e}}\right)^n\})\leq \pi H^2.
\end{equation}
\end{proposition}

Another standard result of complex analysis is the following formula due to Poisson: Let $f(z)$ be a holomorphic function in the disk $\{ |z|\leq R\}$. Write $f(z)=u(z)+iv(z)$ then, for every $z$ such that $|z|<R$, we have:
\begin{equation}\label{poisson}
f(z)={{1}\over{2\pi}}\cdot\int_0^{2\pi}u(Re^{i\theta})\cdot{{Re^{i\theta}+z}\over{Re^{i\theta}-z}}\cdot d\theta+iv(0).
\end{equation}

We will first prove Theorem \ref{main estimate} in the case when $U$ is the disk $\Delta_r:=\{ |z|<r\}$ inside the disk $\Delta_1:=\{ |z|<1\}$ and then a topological argument will allow to deduce the general case. In this case we may suppose that $z_0=0$. 

We would like to remark that, strictly speaking (modulo some adaptations), only the case of disks is necessary for the proofs of this paper, but for sake of completeness, and for future reference, we prefer  giving the proof for a general Riemann Surface. 

The first three lemmas are from complex analysis:

\begin{lemma}\label{harmonic estimate} Let $f(z)$ be an holomorphic function on $\Delta_1$. Let $r<R<1$ be a real number. Write $f(z)=u(z)+iv(z)$ and denote by $A_R(f):=\sup\{u(z)
;/\; |z|<R\}$. Then the following inequality holds:
\begin{equation}
\sup_{\Delta_r}\{ |f(z)|\}\leq [A_R(f)-u(0)]{{2r}\over{R-r}}+|f(0)|.
\end{equation}
\end{lemma}
\begin{proof} Since $u(z)$ is harmonic, we have $u(0)={{1}\over{2\pi}}\int_0^{2\pi} u(Re^{i\theta})d\theta$. From this we obtain
\begin{equation}
f(z)={{1}\over{\pi}}\int_0^{2\pi}u(Re^{i\theta}){{z}\over{Re^{i\theta}-z}}d\theta+f(0).
\end{equation}
In particular, applying this to $f(z)=1$ we obtain ${{1}\over{\pi}}\int_0^{2\pi}{{z}\over{Re^{i\theta}-z}}d\theta=0$. From this we obtain
\begin{equation}
-f(z)={{1}\over{\pi}}\int_0^{2\pi}[A_R(f)-u(Re^{i\theta})]{{z}\over{Re^{i\theta}-z}}d\theta+f(0).
\end{equation}
Since, in the integrant, $A_R(f)-u(Re^{i\theta})\geq 0$ the conclusion follows. \end{proof}

As a consequence, we obtain this non trivial estimate for non vanishing holomorphic functions on a disk:

\begin{lemma}\label{holomorphic estimate}  Let Let $0<r<R<1$ real numbers. Let $f(z)$ be a holomorphic function in $\Delta_1$ non vanishing in $\overline\Delta_R$. Then, for every $z\in\Delta_r$ we have
\begin{equation}
\ln|f(x)|\geq - {{2r}\over{R-r}}\sup_{\Delta_R}\{\ln|f(z)|\}+{{R+r}\over{R-r}}\ln|f(0)|.
\end{equation}
\end{lemma}
\begin{proof} We first remark that we can suppose that $f(0)=1$.

Since $f$ is non vanishing in $\overline\Delta_R$ we may find a neighborhood of it  and a holomorphic function   $g(z)=u(z)+iv(z)$ over this neighborhood, such that $f(z)=e^{g(z)}$. Moreover, since $f(0)=1$ then $g(0)=0$. It is easy to see that $\ln|f(z)|=u(z)$ and consequently $\sup_{\Delta_R}\{\ln|f(z)|\}=A_R(g)$ (notation as in the previous lemma).  If we apply Lemma \ref{harmonic estimate} to the function $g(z)$ we obtain
\begin{equation}
-u(z)\leq {{2r}\over{R-r}}A_R(g).
\end{equation}
From this the conclusion follows. \end{proof}

We generalize now Lemma \ref{holomorphic estimate} to a general holomorphic function on the disk. This can be seen as a simplified  version of Theorem \ref{main estimate} to the case when $M$ is a disk, $U$ is a smaller disk and $L$ is the trivial line bundle with the trivial metric. 

\begin{lemma}\label{holomorphic estimate 2} Let  $0<r<1/3$ be a real number. Let $f(z)$ be a holomorphic function on $\Delta_1$ such that $f(0)\neq 0$. Let $\eta$ be a positive real number less or equal than $1$.  Then the following estimate holds:
\begin{equation}
\mu\left\{ z\in \Delta_r\;/\; \ln|f(z)|<-(2+{{1}\over{\ln(3/2)}}\ln\left({{1}\over{\eta}}\right))\sup_{w\in\Delta_{3r}}\{\ln|f(w)|\} +3\ln|f(0)|\right\}\leq 4\pi e^2\eta^2.
\end{equation}
\end{lemma} 

\begin{proof} Once again it suffices to treat the case when $f(0)=1$.  Let $a_1,\dots, a_n$ be the set of zeros of $f(z)$ inside the closed disk of radius $2r$, counted with multiplicity.  By Jensen formula we have that $n\leq {{1}\over{\ln(3/2)}}\sup_{\Delta_{3r}}\{|f(w)|\}$.

In the sequel of this proof we will denote by $R$ the real number $3r$.  Let 
\begin{equation}
\phi(z):={{(-R)^n}\over{a_1\cdots a_n}}\cdot\prod_{i=1}^n{{R(z-a_i)}\over{R^2-z\overline{a_i}}}.
\end{equation}
Observe that $\phi(0)=1$.  If $f(0)\neq 1$, change $\phi(z)$ with $\phi(z)\cdot f(0))$. The function $g(z):={{f(z)}\over{\phi(z)}}$ is holomorphic on $\Delta_{2r}$ and has no zeros on the closed disk $\overline\Delta_1$. Consequently we can apply Lemma \ref{holomorphic estimate} to it  and find that, for every $z\in \Delta_r$ we have:
\begin{equation} 
\ln|g(z)|\geq -2\sup_{w\in\Delta_{2r}}\{|g(w)|\}.
\end{equation}
Since by hypothesis,  $|\phi(z)|>1$ if $|z|=R$, then for the maximum modulus principle, for every $z\in \Delta_r$,
\begin{equation}
\ln|f(z)|\geq -2\sup_{w\in\Delta_{R}}\{|f(w)|\}+\ln|\phi(z)|
\end{equation}

For every $z\in\Delta_r$ we have that $|R^2-\overline{a}z|\leq 2R^2$ and, by hypothesis, $2|a_i|\leq 1$; thus $|\phi(z)|\geq \prod_{i=1}^n|z-a_i|$. Consequently, for every $z\in\Delta_r$ we have
\begin{equation}\label{lower estimate 1}
\ln|f(z)|\geq -2\sup_{w\in\Delta_{R}}\{|f(w)|\}+\ln\prod_{i=1}^n|z-a_i|.
\end{equation}
Let $z\in\Delta_r$ such that
\begin{equation}\label{small estimate1}
\ln|f(z)|\leq -\left(2+{{1}\over{\ln(2/3)}}\ln\left({{1}\over{\eta}}\right)\right)\sup_{w\in\Delta_{R}}\{|f(w)|\};
\end{equation}
then, since $\eta<1$,
\begin{equation}\label{lower estimate 2}
\ln|f(z)|\leq -\left(2\cdot\sup_{w\in\Delta_{R}}\{|f(w)|\}+n\ln\left({{1}\over{\eta}}\right)\right);
\end{equation}
As a consequence of \ref{lower estimate 1} and \ref{lower estimate 2} we find that, if $z\in\Delta_r$ satisfy the estimate \ref{small estimate1} then $z$ is contained in the set
\begin{equation}
\left\{ w\in \Delta_r\;/\; \prod_{i=1}^n|z-a_i|\leq \eta^n\right\}.
\end{equation}
Since, by Proposition \ref{bloch cartan2}, the area of this last set is bounded above by $4e^2\pi\eta^2$, the Lemma follows. \end{proof}

Following the same proof, Lemma \ref{holomorphic estimate 2} can be improved to an arbitrary holomorphic function on the unit disk: the non vanishing at the origin condition can be removed: if $f(z)$ is an holomorphic function on $\Delta_1$, we can write it as $f(z)=z^ih(z)$ with $h(0)\neq 0$. 

\begin{lemma}\label{holomorphic estimate 3} Let $r$ be as in Lemma \ref{holomorphic estimate 2}. Let $f(z)=z^ih(z)$ be a holomorphic function on $\Delta_1$ (with $h(0)\neq 0$). Let $\eta>0$ and $A$ as in Lemma \ref{holomorphic estimate 2}. Then the following estimate holds:
\begin{equation}
\mu\left\{ z\in \Delta_r\;/\; \ln|f(z)|<-\left(2+{{1}\over{\ln(3/2)}}\ln\left({{1}\over{\eta}}\right)\right)\sup_{w\in\Delta_{3r}}\{\ln|f(w)|\} +3\left(\ln|h(0)|+ i\cdot \ln (3r)\right)\right\}\leq 4\pi e^2\eta^2.
\end{equation}
\end{lemma}
\begin{proof} It suffices to change the $\phi(z)$ of the proof of  Lemma \ref{holomorphic estimate 2} with the function $\psi(z):=\phi(z)\cdot{{z^i}\over{(3r)^i}}$ and proceed as  in that proof. \end{proof}

We can now give the proof of Theorem \ref{main estimate}.

\begin{proof} {\it (of Theorem \ref{main estimate})}. We deal first with the case when the universal covering is $\Delta_1$. We can choose a universal covering $\alpha: \Delta_1\to M$ such that $\alpha(0)\in U$.

Standard formulas give that $d_{\Delta_1}(z,0)< \ln(2)$ if and only if $|z|< 1/3$. 
Consequently, since $U$ is small, We can find a compact closed set $B\in \alpha^{-1}(\overline U)$ having diameter (in the hyperbolic metric) strictly smaller then $\ln(2)$ and such that $\alpha|_B:B\to \overline U$ is surjective. In particular we can find $r<1/3$ such that $B\subset \Delta_{r}$. Denote by $W_0$ the subset $\alpha|_B^{-1}(W)\subset B$.

Observe that $\alpha^\ast(L)$  is isomorphic to the  trivial line bundle on $\Delta_1$. Thus $\alpha^\ast(L)=e\cdot\cO_{\Delta_1}$. Thus we can find positive constants $B_i$ such that, for every $z\in\Delta_{3/5}$, we have $B_1\leq \Vert e\Vert(z)\Vert\leq B_2$.

Let $s\in H^0_B(M,L^d)$ and let $w_0\in W_0$ a point such that $\Vert s\Vert_W=\Vert \alpha^\ast(s)\Vert(w_0)$.  We can choose a automorphism $\iota$ of $\Delta_1$, such that $\iota(0)=w_0$. Since the hyperbolic metric is invariant by the disk automorphisms, we can still find $r_1<1/3$ such that  $\iota^{-1}(B)\subset\Delta_{r_1}$.

Consequently since, for every measurable set $H\subset U$ we have that $\mu(H)\leq \mu((\iota\circ\alpha)^{-1}(H)\cap\iota^{-1}(B))$, we may suppose that $w_0=0$ it suffices to  give an upper estimate for the measure of the set of elements $z\in \Delta_{1/3}$ such that 
\begin{equation}
\ln\Vert \alpha^\ast(s)\Vert(z)\leq -C_1(\ln\left({{1}\over{\eta}}\right)+1)(\log^+\Vert s\Vert_M+d) +3 \ln\Vert\alpha^\ast(s)\Vert(0).
\end{equation}

Write $\alpha^{\ast}(s)=f\cdot e^d$ with $f$ a holomorphic function on $\Delta_1$. Consequently, since $z$ and $w_0$ are elements of $ \Delta_{1/3}$, we may find constants $A_i$, not depending on $z$,  $w_0$ or $s$,  such that
$\ln\Vert\alpha^\ast(s)\Vert (z)\geq \ln|f|(z)+A_1d$ and $\ln\Vert\alpha^\ast(s)\Vert (0)\leq \ln|f|(0)+A_2d$.  Moreover we can find a constant $A_3$ (independent on $z$, $w_0$ or $s$) such that
$\ln^+\Vert s\Vert_M\geq \ln|f|_{3r}+A_3\cdot d$.  

Consequently, choosing $C_1$ sufficiently big,  the Theorem reduces to the estimate of Lemma \ref{holomorphic estimate 2} and the conclusion follows in this case. 

When the universal covering $\tilde M$ is $C$, choose again a compact subset $B\subset\alpha^{-1}(\overline U)$ such that $\alpha|_B:B\to \overline U$ is surjective. Let $h$ be the diameter of $B$ (for the standard flat metric on $\C$). Choose a  disk $\Delta$ of radius bigger then $3h$ centered in a point of $B$ and proceed as before replacing $\Delta_1$ by the chosen disk $\Delta$. \end{proof}

\subsection{proof of Theorem \ref{S sets}} We start with the following easy observation: suppose that $B_1\subset M$ is a subset such that $\varphi(B_1)\cap S(X_K)\neq \emptyset$, then $B_1$ is a subset of type $S$ with respect to $\varphi$. It  suffices then to prove the following:

\begin{claim} If $U$ is a small open set in $M$ containing a subset $B_U$ of type $S$ with respect to $\varphi$ then $\varphi^{-1}(S(X_K))\cap U$ is full in $U$. \end{claim}

Indeed, by hypothesis, we can choose a small open $U_0$ set containing $B$ and $\varphi^{-1}(S(X_K))\cap U_0$ will be full in $U_0$, in particular it will be dense in $U_0$. Choose a covering $\cU$ of  $M$ made by countably many small open sets. For each open set $U$ belonging to $\cU$ and such that $U\cap U_0\neq\emptyset$, we will have that  $\varphi^{-1}(S(X_K))\cap U$ is full on $U$. For each open set $U_1\in \cU$   intersecting one of the open sets which intersect $U_0$, we will have that$\varphi^{-1}(S(X_K))\cap U_1$ is full in $U_1$ etc.  Consequently, by connectedness of $M$, the conclusion follows. 

\begin{proof} {\it (Of the Claim)}. Let $d\geq 1$ and let $s\in H^0(\cX, \cL^d)\setminus\{ 0\}$. Since, $B$ is of type $S$ with respect to $\varphi$, we can find positive constants $C_2$ and $a$ such that
$\log\Vert s\Vert_B\geq -C_2(\log^+\Vert s\Vert +d)^a$.

Choosing the constants $C_i$ suitably  and applying Proposition \ref{main estimate} (and the fact that $B$ is of type $S$ with respect to $\varphi$) we find that the area $V(s)$ of the set 
\begin{equation}
\left\{ z\in U\;\;/\;\; \log\Vert \varphi^\ast(s)\Vert(z)\leq-C_3(\log^+\Vert s\Vert +d)^{\sup\{a, N+1\}}\right\}
\end{equation} is smaller then ${{C_4}\over{{C_5^{d^{N+1}}\cdot (\Vert s\Vert^+)^{C_6\cdot d^N}}}}$.

Consequently we find the following estimate:

\begin{eqnarray*}
&&\sum_{d=1}^\infty\sum_{s\in H^0(\cX;\cL^d)}V(s) \\
&&\leq \sum_{d=0}^\infty\sum_{s\in H^0(\cX;\cL^d)} {{C_4}\over
{(\Vert s\Vert^+)^{C_6d^{n}}C_5^{d^{n+1}}}} \;\;\;\;\;\;\;\text {by  the estimate above}\\
&& =\sum_{d=1}^\infty\sum_{N=1}^\infty\sum_{  {s\in H^0(\cX;\cL^d)}\atop
{N\leq \Vert s\Vert_+<N+1}}  {{C_4}\over
{(\Vert s\Vert^+)^{C_6d^{n}}C_5^{d^{n+1}}}}\\
&&\leq \sum_{d=1}^\infty\sum_{N=1}^\infty{{C_4C_3^{d^{n+1}}N^{d^n}}\over{N^{C_6d^n}C_5^{d^{n+1}}}}\;\; \;\;\;\;\;\;\;\;\;\;\;\;\;\;\text {by formula \ref{eq:counting1}}.
\end{eqnarray*}

The last series converges as soon as we take $C_5$ and $C_6$ sufficiently big. Thus we may apply the Borel--Cantelli Lemma  \ref{borelcantelli}  and the conclusion of the Theorem follows.  \end{proof}

\

\section{Some concluding remarks}\label{conclusion}

\

It is now evident that there is a strict relationship between rational and $S$ points contained in a Riemann surface inside a projective variety. Essentially this relationship may be resumed in the slogan: "Many rational points imply no $S$ points and a single $S$ point implies many $S$ points and few rational points". Moreover, even when we know that the growth of rational point s of height less then $T$ Is just sub exponential (and not less than that), we know that "for may values of $T$" the growth is just polynomial. 

From our point of view, a consequence of the results found in this paper is that the distribution of rational points on analytic sub varieties of a projective variety is still very mysterious and many natural questions arise. For instance, a  natural list of question could be:

a) Suppose that $A_U(T)$ grows not less than sub exponentially, can we estimate how big is the set of $T$'s for which $A_U(T)$ is "big"?

Theorem \ref{rare distribution 1} proved that, even if $A_U(T)$ grows not less then sub exponentially,  we can find intervals as big as we want for which $A_U(T)$ is bounded by a fixed polynomial in $T$. Given a polynomial  $P(t)=C\cdot t^\delta+\cdots$, how wide can be an interval for which $A_U(T)\geq P(T)$? It is possible that, for high values of $\delta$, these intervals are quite small.

b) Can we prove a higher dimensional analogue of Theorem \ref{rare distribution 1}?

As quoted in the introduction, Pila and Wilkie proved an estimate for the analogue of the function $A_U(T)$ for definable sets in a $o$--minimal structure. Roughly speaking, given a definable set $X$, if we remove from it an obvious "algebraic part" (over which there are, most likely, too many rational points"), the set of rational points of the remaining part - called the "transcendental part of $X$ - satisfy an estimate similar to the Bombieri - Pila's. One may expect that, for many values (as in Theorem \ref{rare distribution 1} for instance) of $T$, the number of rational points in the transcendental part of $X$ is just polynomial in $T$.

c) Can we characterize a priori analytically and arithmetically Riemann surfaces which contain few (a lot) of rational points? and same questions for those which intersect the set $S(X_K)$?

Riemann surfaces which are almost all contained in $S(X_K)$ are generic, in many senses. They contain few rational points. One expect that Riemann surfaces which contain many rational points are in some way special so it is possible that one can detect them by some general property which can be computed. Known examples of these surfaces are very artificial, Theorems \ref{foliations} and \ref{foliations and S points} tell us that other examples cannot be constructed as solutions of differential equations, thus at the moment we have a lack of methods to construct them. 

d) Can we detect the structure of $S$ points on a variety from its geometry?

The classical conjecture by Lang asserts that varieties of general type should contain very few rational points (the set of rational points shouldn't be Zariski dense). Is it possible that the set of $S$ points on varieties of general type has some special and peculiar to varieties of general type properties? This would be in the spirit of the slogan quoted at the beginning of this section. We would like to observe that the introduction of small open sets in sub section \ref{section small sections} seems to go in this direction: Theorem \ref{S sets} requires  small open sets and at the moment we do not know if, in the hyperbolic case, the "smallness" hypothesis is necessary. 

\


\begin{thebibliography}{99}

 \bibitem{Byniamini0}   {\sc Binyamini, Gal} \emph{Zero counting and invariant sets of differential equations}, to appear in IMRN.

 \bibitem{Byniamini}   {\sc Binyamini, Gal}, \emph{Multiplicity estimates: a Morse-theoretic approach}. Duke Math. J. 165 (2016), no. 1, 95Ð128. 

 \bibitem{BOMBIERIPILA} {\sc Bombieri, E.. Pila. J.} \emph{The number of rational points on arcs and ovals}, Duke Math. J. 59 (1989), 337--357.
 
 \bibitem{CHUDNOVSKY}{\sc Chudnovsky, G. V.}
 \emph{Contributions to the theory of transcendental numbers.}
Mathematical Surveys and Monographs, 19. American Mathematical Society, Providence, RI, 1984. xi+450 pp.

 \bibitem{comte yomdin} {\sc Comte G., Yomdin Y.} \emph{Zeroes and rational points of analytic functions}, To appear on the Annals of the Ins. Fourier. 

 
\bibitem{gasbarri} \sc Gasbarri, C\emph{ Dyson's theorem for curves}. J. Number Theory 129 (2009), no. 1, 36--58.

\bibitem{gasbarri1} Gasbarri, C. \emph{Analytic subvarieties with many rational points.} Math. Ann. 346 (2010), no. 1, 199--243.

\bibitem{gasbarri2} \sc Gasbarri, C., \emph{Transcendental Liouville inequalities on projective varieties.} Preprint available at https://arxiv.org/abs/1609.04262.

\bibitem{GILLETSOULEISTRAEL}  {\sc Gillet, H.; Soul\'e, C.} \emph{ On the number of lattice points in convex symmetric bodies and their duals.} Israel J. Math. 74 (1991), no. 2-3, 347-357.



\bibitem{Lang} S. Lang, {\it Introduction to complex hyperbolic spaces}, Springer Verlag, New York (1987), 271pp. 

\bibitem{Nesterenko}  Nesterenko, Yu. V. \emph{ Estimates for the number of zeros of certain functions} New advances in transcendence theory (Durham, 1986), 263--269, Cambridge Univ. Press, Cambridge, 1988.

\bibitem{pilawilkie} Pila, J. and Wilkie, A. J., {\it The rational points of a definable set} , Duke Math. J. 133 (2006), 591--616.


\bibitem{SABK} {\sc Soul\'e, C.} \emph{ Lectures on Arakelov geometry. With the collaboration of D. Abramovich, J.-F. Burnol and J. Kramer.} Cambridge Studies in Advanced Mathematics, 33. Cambridge University Press, Cambridge, 1992. viii+177 pp. 

\bibitem{surroca} {\sc A. Surroca,} \emph{Valeurs alg\'ebriques de fonctions transcendantes}. International Mathematics Research Notices, Pages 1--31, 2006.


\bibitem{vanderpoorten}Van der Poorten, A. J. {\it On the number of zeros of functions} Enseignement Math. (2) 23 (1977), no. 1--2, 19--38.

\bibitem{Zhang} {\sc Zhang, S.} \emph{ Positive line bundles on arithmetic varieties.} J. Amer. Math. Soc. 8 (1995), no. 1, 187-221.


\end{thebibliography}
\end{document}